\documentclass[smallextended]{svjour3_arxiv}       
\smartqed  
\usepackage{graphicx}
\usepackage{graphicx}
\usepackage{float}
\usepackage{mathtools}
\usepackage{wrapfig}
\usepackage{amsfonts, amssymb, amsmath, latexsym}
\usepackage{nicefrac}
\usepackage{hhline}
\usepackage{multirow}
\usepackage{pifont}
\usepackage[colorinlistoftodos,bordercolor=orange,backgroundcolor=orange!20,linecolor=orange,textsize=scriptsize]{todonotes}
\usepackage[colorlinks=true,linkcolor=blue,citecolor=blue]{hyperref}       
\usepackage{nicefrac}       
\usepackage{nameref}
\usepackage{booktabs}       

\usepackage{algorithm}
\usepackage{algpseudocode}

\usepackage{xcolor, colortbl}
\usepackage{etoolbox}

\newcommand{\Exp}{\mathbb{E}}

\newcommand{\eps}{\varepsilon}
\renewcommand{\epsilon}{\varepsilon}

\newcommand{\la}{\langle}
\newcommand{\ra}{\rangle}

\newcommand{\lp}{\left(}
\newcommand{\rp}{\right)}

\newcommand{\ls}{\left[}
\newcommand{\rs}{\right]}
\newcommand{\lv}{\left|}
\newcommand{\rv}{\right|}

\newcommand*{\affaddr}[1]{#1} 
\newcommand*{\affmark}[1][*]{\textsuperscript{#1}}

\spnewtheorem{assumption}{Assumption}{\bfseries}{\itshape}

\begin{document}
\title{Zeroth-order methods for noisy H\"older-gradient functions
}


\author{
    Innokentiy Shibaev\affmark[2,1] \and
    Pavel Dvurechensky\affmark[3,2] \and
    Alexander Gasnikov\affmark[1,2,3] 
}

\authorrunning{I. Shibaev, P. Dvurechensky, A. Gasnikov} 

\institute{
    I. Shibaev \at
    \email{innokentiy.shibayev@phystech.edu}
    \and
    P. Dvurechensky \at
    \email{pavel.dvurechensky@wias-berlin.de}
    \and
    A. Gasnikov \at
    \email{gasnikov@yandex.ru}
    \and
    \affaddr{\affmark[1]Moscow Institute of Physics and Technology, Moscow, Russia}\\
    \affaddr{\affmark[2]HSE University, Moscow, Russia}\\
    \affaddr{\affmark[3]Weierstrass Institute for Applied Analysis and Stochastics, Berlin, Germany}%
}


\maketitle

\begin{abstract}
In this paper, we prove new complexity bounds for zeroth-order methods in non-convex optimization with inexact observations of the objective function values. We use the Gaussian smoothing approach of \cite{NesterovSpokoiny2015} and extend their results, obtained for optimization methods for smooth zeroth-order non-convex problems, to the setting of minimization of functions with H\"older-continuous gradient with noisy zeroth-order oracle, obtaining noise upper-bounds as well. We consider finite-difference gradient approximation based on normally distributed random Gaussian vectors and prove that gradient descent scheme based on this approximation converges to the stationary point of the smoothed function. We also consider convergence to the stationary point of the original (not smoothed) function and obtain bounds on the number of steps of the algorithm for making the norm of its gradient small. Additionally we provide bounds for the level of noise in the zeroth-order oracle for which it is still possible to guarantee that the above bounds hold. We also consider separately the case of $\nu = 1$ and show that in this case the dependence of the obtained bounds on the dimension can be improved. 

\keywords{gradient-free methods \and zeroth-order optimization \and non-convex problem \and inexact oracle}
\end{abstract}

\section{Introduction}
The main advantage of zeroth-order (derivative-free) optimization methods \cite{rosenbrock1960automatic,fabian1967stochastic,brent1973algorithms,spall2003introduction,conn2009introduction,larson2019derivative-free} is that computing function value is, in general, simpler than computing its gradient vector. On the one hand, zeroth-order methods usually have worse convergence rates, and may be inferior to gradient methods endowed with Fast Automatic Differentiation (FAD) technique, for which it is known \cite{kim1984efficient,baydin2018automatic} that if there is a series of computational operations to evaluate the value of a function, then with at most four times large number of arithmetic operations it is possible to evaluate the gradient of this function.
On the other hand, there are still a number of situations, when the objective is given as a black-box, there is no access to function derivatives and the FAD technique is not applicable. 
One of the many important recent examples is Reinforcement Learning problems, where the goal is to find an optimal control strategy by observing, in a stochastic environment, some black-box reward function values, see \cite{sutton2018reinforcement} for a review and examples.
The problem can be even more complicated when one deals with computer simulation of some physical processes, e.g. satellite movement, since such models often have some noise in their outputs. Similarly, in Reinforcement Learning only noisy observations of the reward function are available. 
Moreover, the noise can be biased \cite{DOI:10.1609/aaai.v34i04.6086} and standard batch averaging may not help.
Thus, it is important to analyze zeroth-order methods in the setting of possibly biased noisy observations of the objective function.

Another important application of zeroth-order methods with noisy observations is min-max or min-min problems, which are particular settings of bi-level optimization problems. For example, in \cite{bolte2020holderian} the authors consider the problem
\[
\min_x \{f(x)=\max_{y}L(x,y)\},
\]
where $f$ has locally H\"older-continuous gradient and only inexact values of $f$ and its gradient are available via inexact solution to the inner maximization problem in $y$. This leads to a non-convex minimization problem with inexact oracle and the authors focus on first-order inexact oracle. Motivated, in particular, by such problems, we consider in this paper the case when only noisy observations of the objective value $f$ are available. Our bounds on the noise help to evaluate what accuracy of the solution to the inner problem is sufficient to solve the outer problem with some desired accuracy.

\textbf{Related works.}
In \cite{NesterovSpokoiny2015}, among other settings, the authors consider minimization of a non-convex function $f$ on $\mathbb{R}^n$ with exact values of $f$ and used the Gaussian smoothing technique with parameter $\mu$ to prove convergence to a stationary point of a smoothed function $f_{\mu}(x)$, which is a uniform approximation to $f(x)$. The main idea is that the smoothed function $f_{\mu}(x)$ has better properties, e.g. it is smooth even if $f$ is non-smooth. 
In the case when $f$ has Lipschitz-continuous gradient, the authors of \cite{NesterovSpokoiny2015}  prove that their method achieves $\Exp\ls\|\nabla f(x_{N})\|^2_{\ast}\rs\leqslant \varepsilon_{\nabla f}$ after $N=O\lp\frac{n}{\varepsilon_{\nabla f}}\rp$ steps with 2 oracle calls in each step. 
When $f$ is Lipschitz-continuous they estimate an appropriate value of the parameter $\mu$ such that the smoothed function $f_{\mu}(x)$ satisfies $|f_{\mu}(x)-f(x)|\leq \epsilon_f$ for all $x \in \mathbb{R}^n$, and prove that in order to obtain $\Exp\ls\|\nabla f_{\mu}(x_{N})\|^2_{\ast}\rs\leqslant \varepsilon_{\nabla f}$ it is sufficient to make $N=O\lp\frac{n^3}{\epsilon_f \varepsilon_{\nabla f}^2}\rp$ steps of their method with 2 oracle calls in each step. 

This technique was later used in the works \cite{ghadimi2013stochastic} (RSGF algorithm) and \cite{ghadimi2013minibatch} (RSPGF algorithm) to build an algorithm which finds so-called $(\varepsilon,\Lambda)$-solution i.e. a point $x$ s.t. $\mathbb{P}\{\|\nabla f(x)\|^2_{\ast}\leqslant \varepsilon_{\nabla f}\}\geqslant 1-\Lambda$, in the case of Lipschitz-gradient function and stochastic oracle $F(x,\xi)$ s.t. $\Exp_{\xi}[F(x, \xi)] = f(x)$. They have shown that to find an $(\varepsilon_{\nabla f},\Lambda)$-solution it is sufficient to make $O\lp C_1\frac{n}{\varepsilon_{\nabla f}}+C_2\frac{n}{\varepsilon_{\nabla f}^2}\rp$ calls to the stochastic zeroth-order oracle (here the constants $C_1,C_2$ depend on $\Lambda$ and other parameters of the problem, such as Lipschitz constant and diameter of the feasible set). 

In the works \cite{berahas2020theoretical,berahas2019global} the authors compare several types of gradient approximations $g(x)$, including  Gaussian smoothing and smoothing based on uniform sampling on the Euclidean sphere, in terms of the number of calls to the inexact zeroth-order oracle $\hat{f}(x)$ for $f$ which guarantees the approximation condition $\|g(x)-\nabla f(x)\|_{\ast}\leqslant \theta \|\nabla f(x)\|_{\ast}$, where $\theta\in[0,1)$. They show that random-directions-based methods lose in theory to the standard finite differences approach, needing more oracle calls to ensure the above approximation condition. However, in the work \cite{liu2018zerothorder} 
 zeroth-order variants of stochastic variance reduction methods called ZO-SVRG are considered and a variant which uses random directions approach in the experiments required less number of oracle calls than the standard finite differences method (ZO-SVRG-Coord), despite having worse theoretical convergence rate. In this paper we do not rely on the above approximation condition, which allows to obtain better complexity bounds for the considered approach based on random directions and Gaussian smoothing.

\textbf{Our contributions.}
The works listed above mainly focus on the setting when the objective $f$ has Lipschitz-continuous gradient. The only paper, which considers non-smooth setting with $f$ being Lipschitz continuous is \cite{NesterovSpokoiny2015}, where the value of the objective $f$ is assumed to be known exactly.
Our main contribution consists in obtaining complexity bounds for zeroth-order methods with inexact values of the objective in the setting of $f$ having H\"older-continuous gradient, i.e. for some $L_{\nu}>0, \nu \in [0,1]$, $\|\nabla f(x) -\nabla f(y)\|_*\leq L_{\nu}\|x-y\|^{\nu}$. This assumption is more general and includes as particular cases the previously considered settings of objectives $f$ with Lipschitz-continuous gradient and objectives $f$ which are differentiable and Lipschitz continuous.
Our approach uses finite-difference gradient approximation based on normally distributed random Gaussian vectors $u$ and we prove that a gradient descent scheme based on this approximation ensures
\begin{align*}
    \min\limits_{k\in\{0, N-1\}}\Exp_{\mathcal{U}}\ls\|\nabla f_\mu(x_k)\|_{\ast}^2\rs\leqslant \varepsilon_{\nabla f}\text{ after }N = O\lp\frac{n^{\frac{(7-3\nu)}{2}}}{\varepsilon_{\nabla f}^{\frac{3-\nu}{1+\nu}}}\rp
\end{align*}
steps. Here $x_k$ are the iterates, $f_\mu(\cdot)$ is a smoothed version of the objective $f$, $\mathcal{U} = (u_0, . . . , u_{N-1})$ is the history of the realizations of random Gaussian vector $u$.
We also consider convergence to a stationary point of the initial (not smoothed) objective function $f$ and prove that this scheme ensures
\begin{align*}
\min\limits_{k\in\{0, N-1\}}\Exp_{\mathcal{U}}\ls\|\nabla f(x_k)\|_{\ast}^2\rs\leqslant \varepsilon_{\nabla f}\text{ after }N = O\lp\frac{n^{2 + \frac{(1-\nu)}{2\nu}}}{\varepsilon_{\nabla f}^{\frac{1}{\nu}}}\rp
\end{align*}
steps, when $\nu \in (0,1]$. For both cases we obtain bounds for the maximum level of noise in the zeroth-order oracle which does not affect the above iteration complexity bounds. The main difference of our work from \cite{NesterovSpokoiny2015} is that we consider the inexact oracle setting,
intermediate smoothness $\nu \in [0,1]$ rather than the cases $\nu \in {0,1}$, which we also cover in a unified manner.
We additionally provide a refined analysis for the case of $\nu = 1$ to achieve the complexity bound $N=O\lp\frac{n}{\varepsilon_{\nabla f}}\rp$ both for $\min\limits_{k\in\{0, N-1\}}\Exp_{\mathcal{U}}\ls\|\nabla f_\mu(x_k)\|_{\ast}^2\rs\leqslant \varepsilon_{\nabla f}$ and $\min\limits_{k\in\{0, N-1\}}\Exp_{\mathcal{U}}\ls\|\nabla f(x_k)\|_{\ast}^2\rs\leqslant \varepsilon_{\nabla f}$, which is similar to the bound in \cite{NesterovSpokoiny2015} for this case.

The rest of the paper is organized as follows. The first section contains necessary definitions and some technical lemmas which extend or improve the corresponding bounds derived in \cite{NesterovSpokoiny2015}. In the second section, we consider a simple gradient descent process with Gaussian-sampling-based finite-difference gradient approximation and obtain complexity bounds for this method in terms of the gradient norm of the smoothed and of the non-smoothed function. We also analyze how the noise in the objective values influences the convergence and what level of inexactness can be tolerated without changing the convergence properties. 

\section{Gaussian smoothing, zeroth-order oracle}
This section provides problem statement, technical preliminaries and properties of the function $f_\mu$ obtained from $f$ by Gaussian smoothing, as well as the gradient of $f_\mu$, and the estimates for the difference between $f$ and $f_{\mu}$ as well as their gradients.

\subsection{Definitions}

We mostly follow the notation in \cite{NesterovSpokoiny2015} and \cite{Dvurechensky2017GradientMW}, where a similar problem was considered from the point of view of inexact first-order oracle.
We start with some definitions from \cite{NesterovSpokoiny2015}. For an $n$-dimensional space $E$, we denote by $E^{\ast}$ its dual space. The value of a linear function $s \in E^{\ast}$ at point $x \in E$ is denoted by $\la s, x\ra$. We endow the spaces $E$ and $E^{\ast}$ with Euclidean norms
\begin{align}\label{def:norm}
\|x\|^2=\la Bx, x\ra,~\forall x\in E~~~\|s\|_{\ast}^2=\la s, B^{-1}s\ra,~\forall s\in E^{\ast},
\end{align}
where $B:E\to E^{\ast}$ is a linear operator s.t.  $B\succ 0$. 

In this paper we consider the problem of the form
\begin{align}\label{opt_problem}
    \min\limits_{x\in E} f(x)    
\end{align}
under the two following  assumptions.
\begin{assumption}\label{asmpt:noisy_oracle}
    The function $f(x)$ is equipped with an \textit{inexact zeroth-order oracle} $\tilde{f}(x,\delta)$ with some $\delta>0$ i.e. there exists $\delta > 0$ and one can calculate $\tilde{f}(x, \delta )\in \mathbb{R}$ satisfying, for all $x \in E$,
	\begin{align}\label{def:inexact_zeroth_order_oracle}
	& |f(x) - \tilde{f}(x, \delta)|\leqslant\delta.
	\end{align}
\end{assumption}
\begin{assumption}\label{asmpt:holder_grad_function}
    The function $f(x)$ is differentiable with H\"older-continuous gradient with some $\nu \in [0, 1]$ and  $L_{\nu}\geqslant 0$ i.e.
	\begin{align}
		\|\nabla f(y) - \nabla f(x)\|_{\ast} \leqslant L_{\nu}\|y-x\|^{\nu},~\forall x,y\in E.
	\end{align}
\end{assumption}
The latter inequality gives a useful inequality
\begin{align}
	f(y) \leqslant  f(x) + \la \nabla f(x), y-x\ra +  \frac{L_{\nu}}{1 +\nu}\|y-x\|^{1 + \nu},~\forall x,y\in E.
\end{align}

Next, we consider the Gaussian smoothed version of $f(x)$ defined in \cite{NesterovSpokoiny2015}.

\begin{definition}\label{def:gaussian_approximation}
	Consider a function $f:E\to\mathbb{R}$.  Its Gaussian approximation $f_{\mu}(x)$  is defined as
	\begin{align}\label{def:gaussian_approximation:f_exp}
		f_{\mu}(x) = \frac{1}{\kappa}\int\limits_{E}f(x+\mu u) e^{-\tfrac{1}{2}\|u\|^2}du,
	\end{align}
	where
	\begin{align}
		\kappa \overset{\text{def}}{=} \int\limits_{E}e^{-\tfrac{1}{2}\|u\|^2}du = \frac{(2\pi)^{n/2}}{[\det B]^{1/2}}.
	\end{align}
\end{definition}

It can be shown, that (see \cite{NesterovSpokoiny2015} Section 2 for details)
\begin{align}
	\nabla f_\mu(x)& =\frac{1}{\kappa} \int\limits_{E}\frac{f(x+\mu u) - f(x)}{\mu}e^{-\tfrac{1}{2}\|u\|^2}Budu \label{def:gaussian_approximation:gradient_as_f_difference_exp}\\
	&=\frac{1}{\kappa} \int\limits_{E}\frac{f(x+\mu u)}{\mu}e^{-\tfrac{1}{2}\|u\|^2}Budu \label{def:gaussian_approximation:gradient_as_f_exp}\\
	\nabla f(x) & = \frac{1}{\kappa}\int\limits_{E}\la \nabla f(x), u\ra e^{-\tfrac{1}{2}\|u\|^2}Budu,
\end{align}
where the latter equality holds when $f(x)$ is differentiable at $x$. If $f$ is differentiable on $E$, then
\begin{align}
	\nabla f_\mu(x)& =\frac{1}{\kappa} \int\limits_{E}\nabla f(x+\mu u)e^{-\tfrac{1}{2}\|u\|^2}du. \label{def:gaussian_approximation:gradient_as_gradient_exp}
\end{align}
The Gaussian approximation of the function $\tilde{f}(x,\delta)$ then takes the form \begin{align}\label{def:gaussian_approximation:inexact_f_exp}
	\tilde{f}_{\mu}(x,\delta) = \frac{1}{\kappa}\int\limits_{E}\tilde{f}(x+\mu u,\delta) e^{-\tfrac{1}{2}\|u\|^2}du.
\end{align}
We also define the following vector which plays the role of the gradient of $\tilde{f}_{\mu}(x,\delta)$
\begin{align}
	\nabla\tilde{f}_\mu(x, \delta) & = \frac{1}{\kappa}  \int\limits_{E}\frac{\tilde{f}(x+\mu u,\delta) - \tilde{f}(x,\delta)}{\mu}e^{-\tfrac{1}{2}\|u\|^2}Budu \label{def:gaussian_approximation:inexact_gradient_as_inexact_f_difference_exp}\\
	& = \frac{1}{\kappa}  \int\limits_{E}\frac{\tilde{f}(x+\mu u,\delta)}{\mu}e^{-\tfrac{1}{2}\|u\|^2}Budu. \label{def:gaussian_approximation:inexact_gradient_as_inexact_f_exp}
\end{align}

For the case of $\delta = 0$ we have $\nabla\tilde{f}_\mu(x, \delta) = \nabla f_\mu(x)$. It is also worth noting that, in general, it is not possible to obtain for $\nabla\tilde{f}_\mu(x, \delta)$ a representation similar to  \eqref{def:gaussian_approximation:gradient_as_gradient_exp} since the function $\tilde{f}(x, \delta)$ is not necessarily differentiable.

\subsection{Basic results}

As shown in Lemma 3 in \cite{NesterovSpokoiny2015}, for $f(x)$ with Lipschitz-continuous gradient, it holds that
\begin{align*}
	\|\nabla f_\mu(x)  - \nabla f(x)\|_{\ast} \leqslant \frac{\mu L_1}{2}(n+3)^{\nicefrac{3}{2}}.
\end{align*}
This result was improved and extended (see A.1 in \cite{berahas2020theoretical}) to the noisy case giving
\begin{align*}
	\|\nabla\tilde{f}_\mu(x, \delta)  - \nabla f(x)\|_{\ast} \leqslant \frac{\delta}{\mu}n^{1/2} + \mu L_{1}n^{1/2}.
\end{align*}
Extending it to the H\"older case we can show the following result.

\begin{lemma} \label{Lm:2:1} 
	Under Assumptions \ref{asmpt:noisy_oracle} and \ref{asmpt:holder_grad_function} it holds that
	\begin{align*}
	\|\nabla\tilde{f}_\mu(x, \delta)  - \nabla f_\mu(x)\|_{\ast} & \leqslant \frac{\delta}{\mu}n^{1/2} \\
	\|\nabla f_\mu(x)  - \nabla f(x)\|_{\ast} & \leqslant \mu^{\nu} L_{\nu}n^{\nu/2},
	\end{align*}
	and, consequently,
	\begin{align}
	 \label{C_grad_diff}
		\|\nabla\tilde{f}_\mu(x, \delta)  - \nabla f(x)\|_{\ast} \leqslant \frac{\delta}{\mu}n^{1/2} + \mu^{\nu} L_{\nu}n^{\nu/2}.
	\end{align}
\end{lemma}
\begin{proof}  \hyperlink{Lm:2:1:proof}{Can be found in Appendix.}
\end{proof}

It can be shown (assuming $f$ is Lipschitz-continuous with constant $L_0$), that $f_{\mu}$ has H\"older-continuous gradient with $\nu=1$ and $L=\frac{n^{1/2}}{\mu}L_0$ (Lemma 2 from \cite{NesterovSpokoiny2015}). Thus, we can obtain in this case that
\begin{align}\label{ineq:holder_bias_zero_delta}
	|f_{\mu}(y)-f_{\mu}(x)-\la \nabla f_{\mu}(x),y-x\ra|\leqslant \frac{L}{2}\|x-y\|^2.
\end{align}
Under more general H\"older condition we can obtain the following inexact version of \eqref{ineq:holder_bias_zero_delta}.

\begin{lemma} \label{Lm:2:2} 
    Under Assumption \ref{asmpt:holder_grad_function} it holds that
	\begin{align*}
	|f_{\mu}(y)-f_{\mu}(x)-\la \nabla f_{\mu}(x),y-x\ra|\leqslant \frac{A_1}{2}\|y-x\|^2 + A_2,
	\end{align*}
	where either 
	\begin{align*}
	    A_1 = \frac{L_{\nu}}{\mu^{1-\nu}}n^{\tfrac{1+\nu}{2}} \text{, } A_2 = 0,
	\end{align*}
	or
	\begin{align*}
	    A_1 = \ls\frac{1}{\hat{\delta}}\rs^{\frac{1-\nu}{1+\nu}}\frac{2L_{\nu}}{\mu^{1-\nu}} \text{, } A_2 = \hat{\delta}L_{\nu}\mu^{1+\nu} \text{ where }   \hat{\delta}>0.
	\end{align*}
\end{lemma}
\begin{proof}  \hyperlink{Lm:2:2:proof}{Can be found in Appendix.}
\end{proof}

One of the most important properties of the smoothed function $f_{\mu}(x)$ is that it provides a uniform approximation for $f$.
For example, when $f$ is Lipschitz-continuous with constant $L_{0}$ it can be shown (see Theorem 1 from \cite{NesterovSpokoiny2015}) that
\begin{align*}
	|f_{\mu}(x)-f(x)|\leqslant \mu L_{0}n^{1/2}.
\end{align*}
For the more general case of H\"older-continuous gradient we obtain the following more general result.
\begin{lemma} \label{Lm:2:3} 
	Under Assumption \ref{asmpt:holder_grad_function} it can be shown that
	\begin{align*}
	|f_{\mu}(x)-f(x)|\leqslant \frac{L_{\nu}}{1+\nu}\mu^{1+\nu}n^{\frac{1+\nu}{2}}.
	\end{align*}
\end{lemma}
\begin{proof}  \hyperlink{Lm:2:3:proof}{Can be found in Appendix.}
\end{proof}

From Lemma \ref{Lm:2:1} we can obtain an upper bound which connects the gradient norm of $f$ and gradient norm of its smoothed approximation $f_{\mu}$. This will be the key to translate the convergence rate for the smoothed function gradient to the convergence rate of the original objective $f$ gradient.

\begin{lemma} \label{Lm:2:4} 
	Under Assumption \ref{asmpt:holder_grad_function} it holds that
	\begin{align*}
	\|\nabla f(x)\|_{\ast}^2 \leqslant 2\|\nabla f_{\mu}(x)\|_{\ast}^2 + 2\mu^{2\nu} L_{\nu}^2 n^{\nu}.
	\end{align*}
\end{lemma}
\begin{proof}  \hyperlink{Lm:2:4:proof}{Can be found in Appendix.}
\end{proof}

In the next section we consider a gradient descent method with gradient replaced with a random gradient estimation
\begin{align}\label{random_noisy_grad_est}
    g_{\mu}(x,u,\delta)=\frac{\tilde{f}(x+\mu u,\delta) - \tilde{f}(x,\delta)}{\mu}Bu,
\end{align}
where $u$ is a Gaussian random vector with mean $0_n$ and identity covariance matrix $I_n$ ($u\sim\mathcal{N}(0,I_n)$). Thus $\Exp_u\ls g_{\mu}(x,u,\delta)\rs = \nabla\tilde{f}_\mu(x, \delta)$. In what follows we need also one technical result about this estimation.
\begin{lemma} \label{Lm:2:5} 
	Under Assumptions \ref{asmpt:noisy_oracle} and \ref{asmpt:holder_grad_function} for the gradient estimation (\ref{random_noisy_grad_est}) it holds that 
    \begin{align*}
        &\Exp_{u}\ls\|g_{\mu}(x,u,\delta)\|^2_{\ast}\rs \leqslant  20(n+4)\|\nabla f_{\mu}(x)\|^2 + \\
        & + 5\lp \frac{4\delta^2}{\mu^2}n + \frac{4L_{\nu}^2}{(1+\nu)^2}\mu^{2\nu}n^{2+\nu} + \frac{\mu^2 A_1^2}{4}(n+6)^3 + \frac{A_2^2}{\mu^2}n\rp,
    \end{align*}
    where $A_1,A_2$ are constants equal to constants $A_1,A_2$ from Lemma \ref{Lm:2:2}.
\end{lemma}
\begin{proof}  \hyperlink{Lm:2:5:proof}{Can be found in Appendix.}
\end{proof}

\section{Convergence rate analysis}

We consider a gradient descent process
\begin{align}\label{gd_process}
    x_{k+1} = x_k - h_k B^{-1}g_{\mu}(x_k,u_k,\delta),
\end{align}
where $u_k$ is normal random vector and $g_{\mu}(x_k,u_k,\delta)$ is defined in (\ref{random_noisy_grad_est}). We will consider two type of convergence -- in the sense of $\|\nabla f(x_k)\|_{\ast}$ and $\|\nabla f_{\mu}(x_k)\|_{\ast}$. We start with proving the following result
\begin{lemma}\label{Lm:3:1} Consider the process (\ref{gd_process}). Under Assumptions \ref{asmpt:noisy_oracle} and \ref{asmpt:holder_grad_function} it can be shown that after $N-1$ iterations of this process 
    \begin{align}
        \label{ineq:grad_exp_before_A_substitution}
        & \min\limits_{k\in\{0, N-1\}}\Exp_{\mathcal{U}}\ls\|\nabla f_\mu(x_k)\|_{\ast}^2\rs\leqslant 
        \frac{320(n+4)A_1(f_{\mu}(x_0) - f^{\ast})}{ND} + \\
        & + \frac{D}{4(n+4)}\lp \frac{4\delta^2}{\mu^2}n + \frac{4L_{\nu}^2}{(1+\nu)^2}\mu^{2\nu}n^{2+\nu} + \frac{\mu^2 (A_1')^2}{4}(n+6)^3 + \frac{(A_2')^2}{\mu^2}n\rp + \nonumber \\
        & + \frac{320(n+4)A_1}{D}\lp A_2 + \frac{\delta^2}{2A_1\mu^2}n + \frac{\delta^2}{\mu^2}n\rp, \nonumber
    \end{align}
    where $\mathcal{U} = (u_0, . . . , u_{N-1})$ is a random vector composed by i.i.d. $\{u_k\}_{k=0}^{N-1}$, $A_1,A_2$ and $A_1',A_2'$ are the independent pair of constants from Lemma \ref{Lm:2:2} and $D\in(0,1]$.
\end{lemma}
\begin{proof}
    From Lemma \ref{Lm:2:1}, Lemma \ref{Lm:2:2} and the fact that $ab\leqslant\frac{Ca^2}{2} + \frac{b^2}{2C}$ where $C>0$, $a=\|y-x\|$ and $b=\frac{\delta}{\mu}n^{1/2}$ we obtain 
    \begin{align*}
        & |f_{\mu}(y)-f_{\mu}(x) - \la \nabla \tilde{f}_{\mu}(x,\delta),y-x\ra| \leqslant \\ 
        & \leqslant |f_{\mu}(y)-f_{\mu}(x) - \la \nabla f_{\mu}(x),y-x\ra| + |\la \nabla \tilde{f}_{\mu}(x,\delta)-\nabla f_{\mu}(x),y-x\ra| \leqslant \\ 
        & \leqslant \frac{A_1}{2}\|y-x\|^2 + A_2 + \frac{\delta}{\mu}n^{1/2}\|y-x\| \leqslant \\ 
        & \leqslant \lp\frac{A_1}{2}+\frac{C}{2}\rp\|y-x\|^2 + A_2 + \frac{\delta^2}{2C\mu^2}n \overset{C=A_1}{=} A_1\|y-x\|^2 + \lp A_2 + \frac{\delta^2}{2A_1\mu^2}n\rp.
    \end{align*}
    Consider a gradient descent process (\ref{gd_process}). Substituting it and (\ref{random_noisy_grad_est}) into the last inequality and taking the expectation in $u_k$ we obtain
    \begin{align*}
        \Exp_{u_k}\ls f_{\mu}(x_{k+1})\rs \leqslant & ~ f_{\mu}(x_k) - h_k\| \nabla \tilde{f}_{\mu}(x_k,\delta)\|^2_{\ast} + h_k^2A_1\Exp_{u_k}\ls\|g_{\mu}(x_k,u_k,\delta)\|^2_{\ast}\rs + \\
        & + \lp A_2 + \frac{\delta^2}{2A_1\mu^2}n\rp.
    \end{align*}
    Now let's use the fact that (from $(a+b)^2\leqslant 2a^2 + 2b^2$)
    \begin{align*}
        \| \nabla f_{\mu}(x)\|^2_{\ast} & \leqslant  2\| \nabla \tilde{f}_{\mu}(x,\delta)\|^2_{\ast} +2 \|\nabla f_{\mu}(x) - \nabla \tilde{f}_{\mu}(x,\delta)\|^2_{\ast} \leqslant  \\
        &\leqslant 2\| \nabla \tilde{f}_{\mu}(x,\delta)\|^2_{\ast} + 2\cdot \frac{\delta^2}{\mu^2}n
    \end{align*}
    thus
    \begin{align*}
        \Exp_{u_k}\ls f_{\mu}(x_{k+1})\rs \leqslant & ~ f_{\mu}(x_k) - \frac{h_k}{2}\| \nabla f_{\mu}(x_k)\|^2_{\ast} + h_k^2A_1\Exp_{u_k}\ls\|g_{\mu}(x_k,u_k,\delta)\|^2_{\ast}\rs + \\
        & + \lp A_2 + \frac{\delta^2}{2A_1\mu^2}n + \frac{\delta^2}{\mu^2}n\rp.
    \end{align*}
    Substituting result of Lemma \ref{Lm:2:5} (we rename constants from this lemma with $A_1',A_2'$ because it is the second pair of constants, and it can be chosen independently from $A_1,A_2$) we obtain
    \begin{align*}
        & \Exp_{u_k}\ls f_{\mu}(x_{k+1})\rs \leqslant f_{\mu}(x_k) - \lp\frac{h_k}{2} - 20(n+4)h_k^2A_1\rp\| \nabla f_{\mu}(x_k)\|^2_{\ast} +  \\
        & + h_k^2A_1 5\lp \frac{4\delta^2}{\mu^2}n + \frac{4L_{\nu}^2}{(1+\nu)^2}\mu^{2\nu}n^{2+\nu} + \frac{\mu^2 (A_1')^2}{4}(n+6)^3 + \frac{(A_2')^2}{\mu^2}n\rp + \\
        & + \lp A_2 + \frac{\delta^2}{2A_1\mu^2}n + \frac{\delta^2}{\mu^2}n\rp.
    \end{align*}
    Let's choose $h = h_k = \frac{D}{80(n+4)A_1}$ where $D\in(0,1]$ then
    \begin{align*}
        & \Exp_{u_k}\ls f_{\mu}(x_{k+1})\rs \leqslant f_{\mu}(x_k) - \frac{D}{320(n+4)A_1}\| \nabla f_{\mu}(x_k)\|^2_{\ast}  +  \\
        & + \frac{5D^2}{A_1(80(n+4))^2}\lp \frac{4\delta^2}{\mu^2}n + \frac{4L_{\nu}^2}{(1+\nu)^2}\mu^{2\nu}n^{2+\nu} + \frac{\mu^2 (A_1')^2}{4}(n+6)^3 + \frac{(A_2')^2}{\mu^2}n\rp + \\
        & + \lp A_2 + \frac{\delta^2}{2A_1\mu^2}n + \frac{\delta^2}{\mu^2}n\rp
    \end{align*}
    and after summing and taking expectations in $\mathcal{U}$ it becomes
    \begin{align*}
        & \Exp_{\mathcal{U}}\ls f_{\mu}(x_{N})\rs \leqslant f_{\mu}(x_0) - \frac{D}{320(n+4)A_1}\sum\limits_{k=0}^{N-1}\Exp_{\mathcal{U}}\ls\|\nabla f_\mu(x_k)\|_{\ast}^2\rs + \\
        & + \frac{5ND^2}{A_1(80(n+4))^2}\lp \frac{4\delta^2}{\mu^2}n + \frac{4L_{\nu}^2}{(1+\nu)^2}\mu^{2\nu}n^{2+\nu} + \frac{\mu^2 (A_1')^2}{4}(n+6)^3 + \frac{(A_2')^2}{\mu^2}n\rp + \\
        & + N\lp A_2 + \frac{\delta^2}{2A_1\mu^2}n + \frac{\delta^2}{\mu^2}n\rp.
    \end{align*}
    Rearranging terms and using the fact that $f^{\ast}\leqslant \Exp_{\mathcal{U}}\ls f_{\mu}(x_{N})\rs$ we finally obtain (\ref{ineq:grad_exp_before_A_substitution}).
    \qed
\end{proof}

And now we will use it to obtain the rate of convergence and noise bounds for two cases.

\subsection{Convergence in the sense of \texorpdfstring{$\|\nabla f(x_k)\|_{\ast}$}{nablaf(xk)}}

\begin{theorem}\label{thrm:convergence_in_the_sense_of_gradient_norm}
     Consider the process (\ref{gd_process}) and Assumptions \ref{asmpt:noisy_oracle} and \ref{asmpt:holder_grad_function}. Suppose we want to ensure 
    \begin{align*}
        \min\limits_{k\in\{0, N-1\}}\Exp_{\mathcal{U}}\ls\|\nabla f(x_k)\|_{\ast}^2\rs\leqslant \varepsilon_{\nabla f}
    \end{align*}
    then it can be shown that with the right choice of the smoothing parameter $\mu$ this inequality holds after
    \begin{align}\label{thrm:convergence_in_the_sense_of_gradient_norm:N_upper_bound}
        N = O\lp\frac{n^{2 + \frac{1-\nu}{2\nu}}}{\varepsilon_{\nabla f}^{\frac{1}{\nu}}}\rp
    \end{align}
    steps of the process (\ref{gd_process}) under the assumption that
    \begin{align}\label{thrm:convergence_in_the_sense_of_gradient_norm:delta_upper_bound}
        \delta < \frac{\mu^{\frac{3+\nu}{2}}}{n^{\frac{3-\nu}{4}}} = O\left(\frac{\eps_{\nabla f}^{\frac{3+\nu}{4\nu}}}{n^{\frac{3+7\nu}{4\nu}}}\right).
    \end{align}
\end{theorem}
\begin{proof}
    We will use Lemma \ref{Lm:2:4} to replace the gradient norm with the gradient norm of the smoothed function and then use Lemma \ref{Lm:3:1} 
    \begin{align*}
        & \min\limits_{k\in\{0, N-1\}}\Exp_{\mathcal{U}}\ls\|\nabla f(x)\|_{\ast}^2\rs \overset{Lem.~\ref{Lm:2:4}}{\leqslant} \min\limits_{k\in\{0, N-1\}}\lp 2\Exp_{\mathcal{U}}\ls\|\nabla f_{\mu}(x)\|_{\ast}^2\rs+ 2\mu^{2\nu}L_{\nu}^2 n^{\nu}\rp \overset{(\ref{ineq:grad_exp_before_A_substitution})}{\leqslant} \\ 
        & \leqslant \frac{640(n+4)A_1(f_{\mu}(x_0) - f^{\ast})}{ND} + \\
        & + \frac{D}{2(n+4)}\lp \frac{4\delta^2}{\mu^2}n +
        \frac{4L_{\nu}^2}{(1+\nu)^2}\mu^{2\nu}n^{2+\nu} + \frac{\mu^2 (A_1')^2}{4}(n+6)^3 + \frac{(A_2')^2}{\mu^2}n\rp + \\
        & + \frac{640(n+4)A_1}{D}\lp A_2 + \frac{\delta^2}{2A_1\mu^2}n + \frac{\delta^2}{\mu^2}n\rp + 2\mu^{2\nu}L_{\nu}^2 n^{\nu}.
    \end{align*}
    As we can see, the best achievable power of $\mu$ is $2\nu$, so we can choose the remaining parameters based on this. Consider the case $A_1 = \frac{L_{\nu}}{\mu^{1-\nu}}n^{\frac{1+\nu}{2}},A_2 = 0$ and $A_1' = \ls\frac{1}{\hat{\delta}}\rs^{\frac{1-\nu}{1+\nu}}\frac{2L_\nu}{\mu^{1-\nu}},A_2' = \hat{\delta}L_\nu\mu^{1+\nu}$ with $\hat{\delta} = (n+6)^{\frac{1+\nu}{2}}$ (this is chosen to equalize powers of $n$ in second term):
    \begin{align}
        \label{ineq:grad_exp_after_A_substitution}
        & \min\limits_{k\in\{0, N-1\}}\Exp_{\mathcal{U}}\ls\|\nabla f(x)\|_{\ast}^2\rs \leqslant 
        \frac{640(n+4)L_{\nu}(f_{\mu}(x_0) - f^{\ast})}{ND\mu^{1-\nu}}n^{\frac{1+\nu}{2}} + \\
        & + \frac{D\mu^{2\nu}}{2(n+4)}\lp \frac{4\delta^2}{\mu^{2+2\nu}}n + \frac{4L_{\nu}^2}{(1+\nu)^2}n^{2+\nu} + L_{\nu}^2(n+6)^{2+\nu} + L_{\nu}^2n(n+6)^{1+\nu}\rp + \nonumber \\ 
        & + \frac{640(n+4)L_{\nu}}{D\mu^{1-\nu}}n^{\frac{1+\nu}{2}}\lp 0 + \frac{\delta^2}{2L_{\nu}n^{\frac{1+\nu}{2}}\mu^{1+\nu}}n + \frac{\delta^2}{\mu^2}n\rp + 2\mu^{2\nu}L_{\nu}^2 n^{\nu}.\nonumber 
    \end{align}
    Now we see only terms with $\mu^{2\nu}$ and terms with $\delta^2$ and some powers of $\mu$. To ease assumptions on $\delta$ we can consider maximum possible $D = 1$. The bound for $\delta$ then has form of $\delta\leqslant \frac{\mu^{\alpha}}{n^\beta}$, where $\alpha = \frac{3+\nu}{2}$ (from the third term, we have $\mu^{2\alpha - (1 - \nu) - 2}$ and we want it to be $\mu^{2\nu}$) and $\beta = \frac{3-\nu}{4}$ to equalize powers of $n$ in the second ($n^{1+\nu}$) and the third ($n^{2+\frac{1+\nu}{2}-2\beta}$) terms (therefore $\delta < \frac{\mu^{\frac{3+\nu}{2}}}{n^{\frac{3-\nu}{4}}}$):
    \begin{align*}
        & \min\limits_{k\in\{0, N-1\}}\Exp_{\mathcal{U}}\ls\|\nabla f(x)\|_{\ast}^2\rs \leqslant 
        \frac{640(n+4)L_{\nu}(f_{\mu}(x_0) - f^{\ast})}{N\mu^{1-\nu}}n^{\frac{1+\nu}{2}} + \\
        & + \frac{\mu^{2\nu}}{2(n+4)}\lp 4\mu^{1-\nu}n^{\frac{\nu - 1}{2}} + \frac{4L_{\nu}^2}{(1+\nu)^2}n^{2+\nu} + L_{\nu}^2(n+6)^{2+\nu} + L_{\nu}^2n(n+6)^{1+\nu}\rp + \\ 
        & + 320(n+4)\mu^{2\nu}\lp
        \mu^{1-\nu}n^{\frac{\nu - 1}{2}} + 2L_{\nu}n^{\nu}\rp + 2\mu^{2\nu}L_{\nu}^2 n^{\nu}
    \end{align*}
    (notice, that $\mu^{1-\nu}\leqslant 1$ because $\mu<1$ as the step of gradient estimation, so we will replace $\mu^{1-\nu}$ with $1$ further). Consider $\mu\leqslant\mu_0 = \lp  M \cdot n^{1+\nu} \rp^{-\frac{1}{2\nu}}\varepsilon_{\nabla f}^{\frac{1}{2\nu}}$ where
    \begin{align*}
        M \cdot n^{1+\nu} = 
        & \frac{4n^{\frac{\nu - 1}{2}} + \frac{4L_{\nu}^2}{(1+\nu)^2}n^{2+\nu} + 2L_{\nu}^2(n+3)(n+6)^{1+\nu}}{4(n+4)} + \\ 
        & + 160(n+4)\lp n^{\frac{\nu - 1}{2}} + 2L_{\nu}n^{\nu} \rp + \mu^{2\nu}L_{\nu}^2 n^{\nu}
    \end{align*}
    (thus $M = O(1+L_{\nu}+L_{\nu}^2)$) and substituting it we obtain
    \begin{align*}
        & \min\limits_{k\in\{0, N-1\}}\Exp_{\mathcal{U}}\ls\|\nabla f(x_k)\|_{\ast}^2\rs\leqslant \frac{640(n+4)n^{(1-\nu)\cdot\frac{1+\nu}{2\nu}}L_{\nu}(f_{\mu}(x_0) - f^{\ast})}{N\cdot M^{-\frac{1-\nu}{2\nu}}\varepsilon_{\nabla f}^{\frac{2-2\nu}{2\nu}}}n^{\frac{1+\nu}{2}} + \frac{\varepsilon_{\nabla f}}{2}.
    \end{align*}
    That means that we need to make
    \begin{align*}
        N = O\lp\frac{n^{1 + (1-\nu)\cdot\frac{1+\nu}{2\nu} + \frac{1+\nu}{2}}}{\varepsilon_{\nabla f}^{\frac{1}{\nu}}}\rp = O\lp\frac{n^{2 + \frac{1-\nu}{2\nu}}}{\varepsilon_{\nabla f}^{\frac{1}{\nu}}}\rp
    \end{align*}
    steps to ensure $\min\limits_{k\in\{0, N-1\}}\Exp_{\mathcal{U}}\ls\|\nabla f_\mu(x_k)\|_{\ast}^2\rs\leqslant \varepsilon_{\nabla f}$. It's only left to substitute $\mu$ into upper bound for $\delta$ to obtain (\ref{thrm:convergence_in_the_sense_of_gradient_norm:delta_upper_bound}).
    \qed
\end{proof}

In case $\nu = 1$ the article \cite{NesterovSpokoiny2015} (Section 7) shows that the upper bound for the expected number of steps is $N = O\lp\frac{n}{\varepsilon^2}\rp$ where $\varepsilon^2 = \varepsilon_{\nabla f}$, while we show $N = O\lp\frac{n^{2}}{\varepsilon_{\nabla f}}\rp$, which is $n$ times worse. This can be improved quite easily using the fact that for this case 
\begin{align*}
    \|\nabla f_{\mu}(y)-\nabla f_{\mu}(x)\|_{\ast} =
    & \left\| \frac{1}{\kappa} \int\limits_{E}\lp \nabla f(y+\mu u) - \nabla f(x+\mu u) \rp e^{-\tfrac{1}{2}\|u\|^2}du \right\|_{\ast} \leqslant \\
    \leqslant & \frac{1}{\kappa} \int\limits_{E}L_{1}\|y-x\| e^{-\tfrac{1}{2}\|u\|^2}du = L_{1}\|y-x\|
\end{align*}
then this inequality can be used to set $A_1 = \frac{L_1}{2}$ and $A_2 = 0$ in (\ref{ineq:grad_exp_before_A_substitution}), so the power of $n$ in the first term will be 1 less and repeating following steps we will obtain $N = O\lp\frac{n}{\varepsilon_{\nabla f}}\rp$. This, however, cannot be easily extended to $\nu < 1$, because of $\|x-y\|^{\nu}$ term (see Lemma \ref{Lm:2:2} proof for details). 

\subsection{Convergence in the sense of \texorpdfstring{$\|\nabla f_{\mu}(x_k)\|_{\ast}$}{nablafmu(xk)}}

The main problem of the previous result is that it doesn't work with $\nu = 0$ (which is normal because we cannot ensure gradient norm convergence when the gradient is only bounded) and convergence becomes infinitely slow when $\nu\to 0$. We will now consider the convergence in the sense of smoothed function gradient norm while keeping functional gap (Lemma \ref{Lm:2:3}) small.

\begin{theorem}\label{thrm:convergence_in_the_sense_of_smoothed_gradient_norm}
     Consider the process (\ref{gd_process}) and Assumptions \ref{asmpt:noisy_oracle} and \ref{asmpt:holder_grad_function}. Suppose we want to ensure 
    \begin{align*}
        |f_{\mu}(x)-f(x)|\leqslant \frac{L_{\nu}}{1+\nu}\mu^{1+\nu}n^{\frac{1+\nu}{2}} &\leqslant \varepsilon_{f} \\
        \min\limits_{k\in\{0, N-1\}}\Exp_{\mathcal{U}}\ls\|\nabla f_\mu(x_k)\|_{\ast}^2\rs&\leqslant \varepsilon_{\nabla f}
    \end{align*}
    where $\varepsilon_{f} \sim \varepsilon_{\nabla f}^{\frac{1+\nu}{2\nu+\alpha}}$ then it can be shown that with the right choice of the smoothing parameter $\mu$ these inequalities hold after
    \begin{align}\label{thrm:convergence_in_the_sense_of_smoothed_gradient_norm:N_upper_bound}
        N = O\lp\frac{n^{\frac{7-3\nu}{2}}}{\varepsilon_{\nabla f}^{\frac{3-\nu}{1+\nu}}}\rp
    \end{align}
    steps of the process (\ref{gd_process}) under the assumption that
    \begin{align}\label{thrm:convergence_in_the_sense_of_smoothed_gradient_norm:delta_upper_bound}
        \delta <  \frac{\mu^{\frac{5-\nu}{2}}}{n^{\frac{3-\nu}{4}}} = O\left(\frac{\eps_{\nabla f}^{\frac{5-\nu}{2(1+\nu)}}}{n^{\frac{13-3\nu}{4}}}\right).
    \end{align}
\end{theorem}
\begin{proof}
    Substituting the same $A_1,A_2$ and $A_1',A_2'$ as in previous proof into (\ref{ineq:grad_exp_before_A_substitution}) we will obtain almost (\ref{ineq:grad_exp_after_A_substitution}) but without the fourth term and with a smaller constant:
    \begin{align*}
        &  \min\limits_{k\in\{0, N-1\}}\Exp_{\mathcal{U}}\ls\|\nabla f(x)\|_{\ast}^2\rs \leqslant 
        \frac{320(n+4)L_{\nu}(f_{\mu}(x_0) - f^{\ast})}{ND\mu^{1-\nu}}n^{\frac{1+\nu}{2}} + \\
        & + \frac{D\mu^{2\nu}}{4(n+4)}\lp \frac{4\delta^2}{\mu^{2+2\nu}}n + \frac{4L_{\nu}^2}{(1+\nu)^2}n^{2+\nu} + L_{\nu}^2(n+6)^{2+\nu} + L_{\nu}^2n(n+6)^{1+\nu}\rp + \nonumber \\ 
        & + \frac{320(n+4)L_{\nu}}{D\mu^{1-\nu}}n^{\frac{1+\nu}{2}}\lp \frac{\delta^2}{2L_{\nu}n^{\frac{1+\nu}{2}}\mu^{1+\nu}}n + \frac{\delta^2}{\mu^2}n\rp.\nonumber
    \end{align*}
    The difference now is that we are not restricted to use $\eps_{\nabla f} \sim \mu^{2\nu}$, because we can select $D$ to balance powers of $\mu$ (there is no fourth term with its invariable $\mu^{2\nu}$). Let's at first consider a case with $\delta = 0$. Suppose that $D=\mu^{\alpha}$, then
    \begin{align*}
        \min\limits_{k\in\{0, N-1\}}\Exp_{\mathcal{U}}\ls\|\nabla f(x)\|_{\ast}^2\rs \leqslant O\lp\mu^{-(1-\nu+\alpha)}\rp +  O(\mu^{2\nu+\alpha}) + 0
    \end{align*}
    thus $\mu^{2\nu+\alpha}\sim \varepsilon_{\nabla f}$ (because like in previous proof we want to bound the second and the third terms with the $\varepsilon_{\nabla f}/2$) and from Lemma \ref{Lm:2:3} we have 
    \begin{align*}
        \varepsilon_{f}\geqslant \frac{L_{\nu}}{1+\nu}\mu^{1+\nu}n^{\frac{1+\nu}{2}} \sim \varepsilon_{\nabla f}^{\frac{1+\nu}{2\nu+\alpha}}.
    \end{align*}
    
    Now, in previous subsection we had $\mu\sim\varepsilon_{\nabla f}^{\frac{1}{2}}$ (for the case of $\nu=1$), so substituting it into Lemma \ref{Lm:2:3} we would obtain $\varepsilon_{\nabla f}\sim \varepsilon_{f}$. So let's just consider this to be our case, then we can obtain $\frac{1+\nu}{2\nu+\alpha} = 1$ which gives us $\alpha = 1-\nu$ (such reasoning combines results from this and previous sections in the case of $\nu = 1$). 
    
    Now, let's set $D = \mu^{1-\nu} < 1$ and $\delta < \frac{\mu^{\frac{5-\nu}{2}}}{n^{\frac{3-\nu}{4}}}$ (the power of $n$ is chosen similar to the previous proof) then 
    \begin{align*}
        & \min\limits_{k\in\{0, N-1\}}\Exp_{\mathcal{U}}\ls\|\nabla f_\mu(x_k)\|_{\ast}^2\rs\leqslant  \frac{320(n+4)L_{\nu}(f_{\mu}(x_0) - f^{\ast})}{N\mu^{2-2\nu}}n^{\frac{1+\nu}{2}} + \\
        & + \frac{\mu^{1+\nu}}{4(n+4)}\lp 4\mu^{3-3\nu}n^{\frac{\nu - 1}{2}} + \frac{4L_{\nu}^2}{(1+\nu)^2}n^{2+\nu} + 2L_{\nu}^2(n+3)(n+6)^{1+\nu}\rp + \\ 
        & + 160(n+4)\mu^{1+\nu}\lp \mu^{1-\nu}n^{\frac{\nu - 1}{2}} + 2L_{\nu}n^{\nu}\rp.
    \end{align*}
    Consider $\mu\leqslant\mu_0 = \lp  M \cdot n^{1+\nu} \rp^{-\frac{1}{1+\nu}}\varepsilon_{\nabla f}^{\frac{1}{1+\nu}} = \frac{1}{n\cdot M^{\frac{1}{1+\nu}}}\varepsilon_{\nabla f}^{\frac{1}{1+\nu}}$ where
    \begin{align*}
        M\cdot n^{1+\nu} = & \frac{ 4n^{\frac{\nu - 1}{2}} + \frac{4L_{\nu}^2}{(1+\nu)^2}n^{2+\nu} + 2L_{\nu}^2(n+3)(n+6)^{1+\nu}}{8(n+4)} + \\
         & + 80(n+4)\lp n^{\frac{\nu - 1}{2}} + 2L_{\nu}n^{\nu}\rp
    \end{align*}
    and substituting it we obtain
    \begin{align*}
        & \min\limits_{k\in\{0, N-1\}}\Exp_{\mathcal{U}}\ls\|\nabla f_\mu(x_k)\|_{\ast}^2\rs\leqslant \frac{320(n+4)n^{2-2\nu}L_{\nu}(f_{\mu}(x_0) - f^{\ast})}{N\cdot M^{-\frac{2-2\nu}{1+\nu}}\varepsilon_{\nabla f}^{\frac{2-2\nu}{1+\nu}}}n^{\frac{1+\nu}{2}} + \frac{\varepsilon_{\nabla f}}{2}.
    \end{align*}
    That means that we need to make
    \begin{align}
        N = O\lp\frac{n^{1 +(2-2\nu) + \frac{1+\nu}{2}}}{\varepsilon_{\nabla f}^{\frac{3-\nu}{1+\nu}}}\rp = O\lp\frac{n^{\frac{7-3\nu}{2}}}{\varepsilon_{\nabla f}^{\frac{3-\nu}{1+\nu}}}\rp
    \end{align}
    steps to ensure $\min\limits_{k\in\{0, N-1\}}\Exp_{\mathcal{U}}\ls\|\nabla f_\mu(x_k)\|_{\ast}^2\rs\leqslant \varepsilon_{\nabla f}$. Substituting $\mu = \mu_0$ into Lemma \ref{Lm:2:3} we obtain 
    \begin{align*}
        |f_{\mu}(x)-f(x)|\leqslant \frac{L_{\nu}}{1+\nu}\mu^{1+\nu}n^{\frac{1+\nu}{2}} = \Theta\lp \frac{\varepsilon_{\nabla f}}{n^{\frac{1+\nu}{2}}}\rp.
    \end{align*} 
    Thus we ensure $|f_{\mu}(x)-f(x)|\leqslant \varepsilon_{f}$ with $\varepsilon_f = \Theta\lp \frac{\varepsilon_{\nabla f}}{n^{\frac{1+\nu}{2}}}\rp$. The bound (\ref{thrm:convergence_in_the_sense_of_smoothed_gradient_norm:delta_upper_bound}) can be obtained the same  way as in previous theorem.
    \qed
\end{proof}

In case $\nu = 0$ \cite{NesterovSpokoiny2015} shows that $N = O\lp\frac{n^3}{\varepsilon_{f}\varepsilon_{\nabla f}^2}\rp \overset{\varepsilon_f = \Theta\lp \frac{\varepsilon_{\nabla f}}{n^{1/2}}\rp}{=} O\lp\frac{n^{\frac{7}{2}}}{\varepsilon_{\nabla f}^3}\rp$ which coincides with our result. In case $\nu = 1$ this result coincides with the result of the previous theorem, and we can repeat the reasoning at the end improving the result by making the iteration complexity to be proportional to $n$ rather than $n^2$.

We didn't discuss the question of what is the weakest possible bound on $\delta$ at which it is still possible to prove the convergence. It can be easily shown that if we remove powers of $n$ from these $\delta$ upper bounds it won't change the fact of the convergence, however this will increase the powers of $n$ in $N$ bounds. For example in the end of the proof of the Theorem \ref{thrm:convergence_in_the_sense_of_smoothed_gradient_norm} we can choose $\mu_0 = \lp  M \cdot n^{\frac{5+\nu}{2}} \rp^{-\frac{1}{1+\nu}}\varepsilon_{\nabla f}^{\frac{1}{1+\nu}}$ (this is the biggest power of $n$ there) and then repeating the steps we obtain 
\begin{align*}
    N = O\lp\frac{n^{1 +\frac{5+\nu}{2}\frac{1}{1+\nu} + \frac{1+\nu}{2}}}{\varepsilon_{\nabla f}^{\frac{3-\nu}{1+\nu}}}\rp.
\end{align*}
Changing the powers of $\eps_{\nabla f}$ for noise bounds is harder though, and can be a topic of the further studies.

\section{Conclusion} \label{section:conclusion}
In this paper we extend the results of \cite{NesterovSpokoiny2015} to non-convex minimization problems with H\"older-continuous gradients and noisy zeroth-order oracle. Table \ref{convergence-table} below summarizes our results for two types of the quality measures: norm of the gradient of the smoothed version of the objective $f_{\mu}$ and norm of the gradient of the original objective function $f$. We provide an upper bound for the necessary number of iterations $N$ and an upper bound on the oracle inexactness $\delta$ which can be tolerated and still allows to achieve the desired accuracy in terms of the corresponding criterion.
We also show that in the case  $\nu = 1$, the upper bounds for $N$  can be improved by reducing the exponent of $n$ to 1 (second part of the Table \ref{convergence-table}). The interesting fact is that for the case of $\nu = 1$ the upper bound for the noise level $\delta$ is linear in $\eps_{\nabla f}$, and bounds on $N$ and $\delta$ for both $\|\nabla f(x_k)\|_{\ast}$ and $\|\nabla f_{\mu}(x_k)\|_{\ast}$ coincide.

In future it would be interesting to explore in more details the trade-off between the oracle noise level $\delta$ and the iteration number $N$ in terms of their dependence on $n$, which we briefly discussed after the proofs of Theorems \ref{thrm:convergence_in_the_sense_of_gradient_norm} and \ref{thrm:convergence_in_the_sense_of_smoothed_gradient_norm}.
Another interesting question for future research is whether it is possible to obtain a bound for $N$ which continuously depends on $\nu$ and for $\nu=1$ gives the same bound as the bound in \cite{NesterovSpokoiny2015}.

\begin{table}[h!]
  \caption{Convergence properties for the different convergence types}
  \label{convergence-table}
  \begin{tabular}{ccccc}
    \toprule
    Convergence type & $N$ upper bound & $\delta$ upper bound & $|f_{\mu}(x)-f(x)|$ & Possible  $\nu$ \\
    \midrule
    $\Exp\|\nabla f(x_k)\|_{\ast}^2 \leqslant  \varepsilon_{\nabla f}$ & $O\lp\frac{n^{2 + \frac{1-\nu}{2\nu}}}{\varepsilon_{\nabla f}^{\frac{1}{\nu}}}\rp$  & $O\left(\frac{\eps_{\nabla f}^{\frac{3+\nu}{4\nu}}}{n^{\frac{3+7\nu}{4\nu}}}\right)$ & ~---  &  $\nu\in(0,1]$  \\
    $\Exp\|\nabla f_{\mu}(x_k)\|_{\ast}^2 \leqslant \varepsilon_{\nabla f}$ & $O\lp\frac{n^{\frac{7-3\nu}{2}}}{\varepsilon_{\nabla f}^{\frac{3-\nu}{1+\nu}}}\rp$ & $O\left(\frac{\eps_{\nabla f}^{\frac{5-\nu}{2(1+\nu)}}}{n^{\frac{13-3\nu}{4}}}\right)$ &  $\Theta\lp \frac{\varepsilon_{\nabla f}}{n^{\frac{1+\nu}{2}}}\rp$ & $\nu\in[0,1]$  \\
    \midrule
    $\Exp\|\nabla f(x_k)\|_{\ast}^2 \leqslant  \varepsilon_{\nabla f}$ & $O\lp\frac{n}{\varepsilon_{\nabla f}}\rp$ & $O\left(\frac{\eps_{\nabla f}}{n^{\nicefrac{5}{2}}}\right)$ & ~---  &  $\nu=1$  \\
    $\Exp\|\nabla f_{\mu}(x_k)\|_{\ast}^2 \leqslant \varepsilon_{\nabla f}$ & $O\lp\frac{n}{\varepsilon_{\nabla f}}\rp$ & $O\left(\frac{\eps_{\nabla f}}{n^{\nicefrac{5}{2}}}\right)$ &  $\Theta\lp \frac{\varepsilon_{\nabla f}}{n}\rp$ & $\nu = 1$  \\
    \bottomrule
  \end{tabular}
\end{table}

\begin{acknowledgements}
    The authors are grateful to K. Scheinberg and A. Beznosikov for several discussions on derivative-free methods.
\end{acknowledgements}

\bibliographystyle{spmpsci}      
\bibliography{main} 

\begin{thebibliography}{10}
\providecommand{\url}[1]{{#1}}
\providecommand{\urlprefix}{URL }
\expandafter\ifx\csname urlstyle\endcsname\relax
  \providecommand{\doi}[1]{DOI~\discretionary{}{}{}#1}\else
  \providecommand{\doi}{DOI~\discretionary{}{}{}\begingroup
  \urlstyle{rm}\Url}\fi

\bibitem{baydin2018automatic}
Baydin, A.G., Pearlmutter, B.A., Radul, A.A., Siskind, J.M.: Automatic
  differentiation in machine learning: a survey.
\newblock arxiv:1502.05767  (2018)

\bibitem{berahas2020theoretical}
Berahas, A.S., Cao, L., Choromanski, K., Scheinberg, K.: A theoretical and
  empirical comparison of gradient approximations in derivative-free
  optimization.
\newblock arxiv:1905.01332  (2019)

\bibitem{berahas2019global}
Berahas, A.S., Cao, L., Scheinberg, K.: Global convergence rate analysis of a
  generic line search algorithm with noise.
\newblock arxiv:1910.04055  (2019)

\bibitem{bolte2020holderian}
Bolte, J., Glaudin, L., Pauwels, E., Serrurier, M.: A h\"olderian backtracking
  method for min-max and min-min problems.
\newblock arxiv:2007.08810  (2020)

\bibitem{brent1973algorithms}
Brent, R.: Algorithms for Minimization Without Derivatives.
\newblock Dover Books on Mathematics. Dover Publications (1973)

\bibitem{conn2009introduction}
Conn, A.R., Scheinberg, K., Vicente, L.N.: Introduction to Derivative-Free
  Optimization.
\newblock Society for Industrial and Applied Mathematics (2009).
\newblock \doi{10.1137/1.9780898718768}

\bibitem{Dvurechensky2017GradientMW}
Dvurechensky, P.: Gradient method with inexact oracle for composite non-convex
  optimization.
\newblock arxiv:1703.09180  (2017)

\bibitem{fabian1967stochastic}
Fabian, V.: Stochastic approximation of minima with improved asymptotic speed.
\newblock Ann. Math. Statist. \textbf{38}(1), 191--200 (1967).
\newblock \doi{10.1214/aoms/1177699070}

\bibitem{ghadimi2013stochastic}
Ghadimi, S., Lan, G.: Stochastic first- and zeroth-order methods for nonconvex
  stochastic programming.
\newblock {SIAM} Journal on Optimization \textbf{23}(4), 2341--2368 (2013).
\newblock \doi{10.1137/120880811}.
\newblock \urlprefix\url{https://doi.org/10.1137/120880811}

\bibitem{kim1984efficient}
Kim, K., Nesterov, Y., Skokov, V., Cherkasskii, B.: Effektivnii algoritm
  vychisleniya proisvodnyh i ekstremalnye zadachi (efficient algorithm for
  calculation of derivatives and extreme problems).
\newblock Ekonomika i matematicheskie metody \textbf{20}(2), 309--318 (1984)

\bibitem{larson2019derivative-free}
Larson, J., Menickelly, M., Wild, S.M.: Derivative-free optimization methods.
\newblock Acta Numerica \textbf{28}, 287–404 (2019).
\newblock \doi{10.1017/S0962492919000060}

\bibitem{liu2018zerothorder}
Liu, S., Kailkhura, B., Chen, P.Y., Ting, P., Chang, S., Amini, L.:
  Zeroth-order stochastic variance reduction for nonconvex optimization.
\newblock Advances in Neural Information Processing Systems \textbf{31},
  3727--3737 (2018)

\bibitem{Nesterov2015UGM}
Nesterov, Y.: Universal gradient methods for convex optimization problems.
\newblock Mathematical Programming \textbf{152}(1), 381--404 (2015).
\newblock \doi{10.1007/s10107-014-0790-0}.
\newblock \urlprefix\url{https://doi.org/10.1007/s10107-014-0790-0}

\bibitem{NesterovSpokoiny2015}
Nesterov, Y., Spokoiny, V.: Random gradient-free minimization of convex
  functions.
\newblock Foundations of Computational Mathematics \textbf{17}(2), 527--566
  (2015).
\newblock \doi{10.1007/s10208-015-9296-2}

\bibitem{rosenbrock1960automatic}
Rosenbrock, H.H.: An automatic method for finding the greatest or least value
  of a function.
\newblock The Computer Journal \textbf{3}(3), 175--184 (1960).
\newblock \doi{10.1093/comjnl/3.3.175}

\bibitem{ghadimi2013minibatch}
Saeed~Ghadimi, G.L., Zhang, H.: Mini-batch stochastic approximation methods for
  nonconvex stochastic composite optimization.
\newblock Mathematical Programming \textbf{155} (2013).
\newblock \doi{10.1007/s10107-014-0846-1}

\bibitem{spall2003introduction}
Spall, J.C.: Introduction to Stochastic Search and Optimization, 1 edn.
\newblock John Wiley \& Sons, Inc., New York, NY, USA (2003)

\bibitem{sutton2018reinforcement}
Sutton, R.S., Barto, A.G.: Reinforcement learning: An introduction.
\newblock MIT press (2018)

\bibitem{DOI:10.1609/aaai.v34i04.6086}
Wang, J., Liu, Y., Li, B.: Reinforcement learning with perturbed rewards.
\newblock Proceedings of the AAAI Conference on Artificial Intelligence
  \textbf{34}, 6202--6209 (2020).
\newblock \doi{10.1609/aaai.v34i04.6086}

\end{thebibliography}

\appendix

\section{Appendix}
\subsection{Proofs of Lemmas 2.1~---2.5}
\hypertarget{Lm:2:1:proof}{}
\begin{proof}[Lemma \ref{Lm:2:1}]
    From (\ref{def:norm}) we get $\|Bu\|_{\ast}^2 = \la Bu,B^{-1}Bu\ra=\la Bu,u\ra=\|u\|^2$.  Using this and Lemma \ref{lemma_1_from_Nesterov_Spokoiny} we obtain
    \begin{align*}
	& \|\nabla\tilde{f}_\mu(x, \delta)  - \nabla f_\mu(x)\|_{\ast} \overset{(\ref{def:gaussian_approximation:inexact_gradient_as_inexact_f_exp})}{=} \\
	& =\left\| \frac{1}{\kappa}\int\limits_{E}\frac{\tilde{f}(x+\mu u,\delta)\pm f(x+\mu u)}{\mu} e^{-\tfrac{1}{2}\|u\|^2}Budu - \nabla f_\mu(x)\right\|_{\ast}\leqslant \\ 
	& \leqslant \left\|\frac{1}{\kappa} \int\limits_{E}\lp\frac{\tilde{f}(x+\mu u,\delta) - f(x+\mu u)}{\mu}\rp  e^{-\tfrac{1}{2}\|u\|^2}Budu \right\|_{\ast} + \\ 
	& + \left\|\frac{1}{\kappa} \int\limits_{E}\frac{f(x+\mu u)}{\mu} e^{-\tfrac{1}{2}\|u\|^2}Budu - \nabla f_\mu(x) \right\|_{\ast}\overset{Asm.~\ref{asmpt:noisy_oracle},~(\ref{def:gaussian_approximation:gradient_as_f_exp})}{\leqslant} \\ 
	& \leqslant \frac{1}{\kappa}\int\limits_{E}\frac{\delta}{\mu} \|u\| e^{-\tfrac{1}{2}\|u\|^2}du +  \left\| \nabla f_\mu(x) - \nabla f_\mu(x) \right\|_{\ast} \overset{Lem.~\ref{lemma_1_from_Nesterov_Spokoiny}}{\leqslant} \frac{\delta}{\mu}n^{1/2} \\
	\end{align*}
	and
    \begin{align*}
	&\|\nabla f_\mu(x)  - \nabla f(x)\|_{\ast} \overset{(\ref{def:gaussian_approximation:gradient_as_gradient_exp})}{=} \left\|\frac{1}{\kappa}\int\limits_{E}\lp\nabla f(x+\mu u) - \nabla f(x) \rp  e^{-\tfrac{1}{2}\|u\|^2}du \right\|_{\ast} \overset{Asm.~\ref{asmpt:holder_grad_function}}{\leqslant} \\ 
	&\leqslant \frac{1}{\kappa}\int\limits_{E}L_{\nu}\|\mu u\|^{\nu}e^{-\tfrac{1}{2}\|u\|^2}du \overset{Lem.~\ref{lemma_1_from_Nesterov_Spokoiny}}{\leqslant} \mu^{\nu} L_{\nu}n^{\nu/2}
	\end{align*}
    thus, finally
	\begin{align*}
	& \|\nabla\tilde{f}_\mu(x, \delta)  - \nabla f(x)\|_{\ast}  \leqslant \|\nabla\tilde{f}_\mu(x, \delta)  - \nabla f_\mu(x)\|_{\ast} + \|\nabla f_\mu(x)  - \nabla f(x)\|_{\ast} \leqslant \\ 
	& \leqslant \frac{\delta}{\mu}n^{1/2} + \mu^{\nu} L_{\nu}n^{\nu/2}.
	\tag*{\qed}
	\end{align*}
    \let\qed\relax
\end{proof}

\hypertarget{Lm:2:2:proof}{}
\begin{proof}[Lemma \ref{Lm:2:2}]
	\begin{align*}
	&\|\nabla f_{\mu}(y)-\nabla f_{\mu}(x)\|_{\ast} \overset{(\ref{def:gaussian_approximation:gradient_as_f_difference_exp})}{=} \\
	& = \frac{1}{\kappa} \left\| \int\limits_{E}\lp\frac{f(y+\mu u) - f(y)}{\mu}-\frac{f(x+\mu u) - f(x)}{\mu}\rp Bu e^{-\tfrac{1}{2}\|u\|^2}du \right\|_{\ast} \leqslant \\ 
	&\leqslant 
	\frac{1}{\mu\kappa} \int\limits_{E}\left| \int\limits_{0}^{1}\la \nabla f(\mu u + ty + (1-t)x) - \nabla f(ty + (1-t)x), y-x\ra dt\right| \|u\| e^{-\tfrac{1}{2}\|u\|^2}du \overset{Asm.~\ref{asmpt:holder_grad_function}}{\leqslant} \\ 
	&\leqslant
	\frac{1}{\mu\kappa} \int\limits_{E}L_{\nu}\mu^{\nu} \|y-x\|\|u\|^{1+\nu} e^{-\tfrac{1}{2}\|u\|^2}du \overset{Lem.~\ref{lemma_1_from_Nesterov_Spokoiny}}{\leqslant}
	\frac{L_{\nu}}{\mu^{1-\nu}}n^{\tfrac{1+\nu}{2}}\|y-x\|.
    \end{align*}
    Integrating this we obtain
	\begin{align}
    f_{\mu}(y)-f_{\mu}(x)-\la \nabla f_{\mu}(x),y-x\ra\leqslant \frac{L_{\nu}}{2\mu^{1-\nu}}n^{\tfrac{1+\nu}{2}}\|y-x\|^2
    \end{align} 
    so using this way we proved lemma with $A_1 = \frac{L_{\nu}}{\mu^{1-\nu}}n^{\tfrac{1+\nu}{2}}$ and $A_2 = 0$.

    The other way to obtain $A_1$ and $A_2$ is to directly upper bound $f_{\mu}(y)-f_{\mu}(x)-\la \nabla f_{\mu}(x),y-x\ra$ applying Lemma \ref{lemma_2_from_Nesterov_2015_UGM}:
    \begin{align*}
    & f_{\mu}(y)-f_{\mu}(x)-\la \nabla f_{\mu}(x),y-x\ra \overset{(\ref{def:gaussian_approximation:f_exp},\ref{def:gaussian_approximation:gradient_as_gradient_exp})}{=} \\
    & = \frac{1}{\kappa}\int\limits_{E}\lp f(y+\mu u) - f(x+\mu u) - \la \nabla f(x+\mu u),y-x\ra\rp e^{-\tfrac{1}{2}\|u\|^2}du \overset{Asm.~\ref{asmpt:holder_grad_function}}{\leqslant} \\ 
    & \leqslant \frac{L_{\nu}}{1+\nu}\|y-x\|^{1+\nu} 
    \overset{\text{Lem.~\ref{lemma_2_from_Nesterov_2015_UGM}}}{\leqslant} \nonumber  \frac{1}{2}\ls\frac{1-\nu}{1+\nu}\frac{2}{\tilde{\delta}}\rs^{\frac{1-\nu}{1+\nu}} L_{\nu}^{\frac{2}{1+\nu}}\|y-x\|^2 + \tilde{\delta}.
    \end{align*}
    Setting $\tilde{\delta} = \hat{\delta}\mu^{1+\nu}L_\nu$ and using upper bound  $\ls 2\frac{1-\nu}{1+\nu}\rs^{\frac{1-\nu}{1+\nu}}\leqslant 2$ we obtain
    \begin{align}\label{lemma_2_2_result_direct}
    & f_{\mu}(y)-f_{\mu}(x)-\la \nabla f_{\mu}(x),y-x\ra \leqslant \ls\frac{1}{\hat{\delta}}\rs^{\frac{1-\nu}{1+\nu}}\frac{L_{\nu}}{\mu^{1-\nu}}\|y-x\|^2 + \hat{\delta}L_{\nu}\mu^{1+\nu}
    \end{align}
    so we proved lemma with $A_1 = \ls\frac{1}{\hat{\delta}}\rs^{\frac{1-\nu}{1+\nu}}\frac{2L_{\nu}}{\mu^{1-\nu}}$ and $A_2 = \hat{\delta}L_{\nu}\mu^{1+\nu}$. 
\end{proof}

\hypertarget{Lm:2:3:proof}{}
\begin{proof}[Lemma \ref{Lm:2:3}]
    To proof this we should notice that 
    \begin{align*}
        \frac{1}{\kappa} \int\limits_{E}\la \nabla f(x), u\ra e^{-\tfrac{1}{2}\|u\|^2}du = 0
    \end{align*}
    thus
    \begin{align*}
        & |f_{\mu}(x)-f(x)| \overset{(\ref{def:gaussian_approximation:f_exp})}{=} \left|\int\limits_{E}\lp f(x+\mu u) - f(x)\rp e^{-\tfrac{1}{2}\|u\|^2}du\right| = \\ 
        & =\left|\int\limits_{E}\lp f(x+\mu u) - f(x) - \la \nabla f(x), \mu u\ra\rp e^{-\tfrac{1}{2}\|u\|^2}du\right| \overset{Asm.~\ref{asmpt:holder_grad_function}}{\leqslant} \\ 
        & \leqslant \frac{L_{\nu}}{1 + \nu}\mu^{1 + \nu}\int\limits_{E}\|u\|^{1+\nu} e^{-\tfrac{1}{2}\|u\|^2}du \overset{Lem.~\ref{lemma_1_from_Nesterov_Spokoiny}}{\leqslant}\frac{L_{\nu}}{1 + \nu}\mu^{1 + \nu}n^{\frac{1+\nu}{2}}.
        \tag*{\qed}
    \end{align*}
    \let\qed\relax
\end{proof}
\hypertarget{Lm:2:4:proof}{}
\begin{proof}[Lemma \ref{Lm:2:4}] From the fact that $a^2 \leqslant 2(a+b)^2 + 2b^2$:
    \begin{align*}
        \|\nabla f(x)\|_{\ast}^2 & \leqslant  2\|\nabla f_{\mu}(x)\|_{\ast}^2 + 2\|\nabla f(x) - \nabla f_{\mu}(x)\|_{\ast}^2 \overset{Lem.~\ref{Lm:2:1}}{\leqslant} \\
        & \leqslant 2\|\nabla f_{\mu}(x)\|_{\ast}^2 + 2\mu^{2\nu}L_{\nu}^2 n^{2\nu}.
        \tag*{\qed}
    \end{align*}
    \let\qed\relax
\end{proof}

\hypertarget{Lm:2:5:proof}{}
\begin{proof}[Lemma \ref{Lm:2:5}]
    \begin{align*}
        \Exp_{u}\ls\|g_{\mu}(x,u,\delta)\|^2_{\ast}\rs = \frac{1}{\kappa}  \int\limits_{E}\lv\frac{\tilde{f}(x+\mu u,\delta) - \tilde{f}(x,\delta)}{\mu}\rv^2 \|u\|^2 e^{-\tfrac{1}{2}\|u\|^2}du
    \end{align*}
    let's bound  $|\tilde{f}(x+\mu u,\delta) - \tilde{f}(x,\delta)|$:
    \begin{align*}
        & |\tilde{f}(x+\mu u,\delta) - \tilde{f}(x,\delta)| \leqslant 2\delta + |f(x+\mu u) - f(x)| \leqslant \\ 
        & \leqslant 2\delta + |f(x+\mu u) - f_{\mu}(x+\mu u) - f(x) + f_{\mu}(x)| + |f_{\mu}(x+\mu u) - f_{\mu}(x)| \overset{Lem.~\ref{Lm:2:3}}{\leqslant} \\ 
        & \leqslant 2\delta + \frac{2L_{\nu}}{1+\nu}\mu^{1+\nu}n^{\frac{1+\nu}{2}} + |f_{\mu}(x+\mu u) - f_{\mu}(x) - \la\nabla f_{\mu}(x),\mu u\ra| + \\ 
        & + |\la\nabla f_{\mu}(x),\mu u\ra| \overset{Lem.~\ref{Lm:2:2}}{\leqslant} \\ 
        & \leqslant 2\delta + \frac{2L_{\nu}}{1+\nu}\mu^{1+\nu}n^{\frac{1+\nu}{2}} + \frac{\mu^2 A_1}{2}\|u\|^2 + A_2 + |\la\nabla f_{\mu}(x),\mu u\ra| 
    \end{align*}
    thus from the fact that $\lp\sum_{i=1}^{k}a_i\rp^2\leqslant k\lp\sum_{i=1}^{k}a_i^2\rp$
    \begin{align*}
        & |\tilde{f}(x+\mu u,\delta) - \tilde{f}(x,\delta)|^2 \leqslant \\
        & \leqslant 5\lp 4\delta^2 + \frac{4L_{\nu}^2}{(1+\nu)^2}\mu^{2+2\nu}n^{1+\nu} + \frac{\mu^4 A_1^2}{4}\|u\|^4 + A_2^2 + \la\nabla f_{\mu}(x),\mu u\ra^2\rp
    \end{align*}
    and applying Theorem \ref{theorem_3_from_Nesterov_Spokoiny} we finally obtain
    \begin{align*}
        &\Exp_{u}\ls\|g_{\mu}(x,u,\delta)\|^2_{\ast}\rs \leqslant  20(n+4)\|\nabla f_{\mu}(x)\|^2 + \\
        & + 5\lp \frac{4\delta^2}{\mu^2}n + \frac{4L_{\nu}^2}{(1+\nu)^2}\mu^{2\nu}n^{2+\nu} + \frac{\mu^2 A_1^2}{4}(n+6)^3 + \frac{A_2^2}{\mu^2}n\rp
        \tag*{\qed}
    \end{align*}
    \let\qed\relax
\end{proof}

\subsection{External results}

\begin{lemma}[Lemma 1 from \cite{NesterovSpokoiny2015}] \label{lemma_1_from_Nesterov_Spokoiny} For $p\geqslant 0$, we have
    \begin{align*}
	\frac{1}{\kappa}\int\limits_{E}\|u\|^{p}e^{-\tfrac{1}{2}\|u\|^2}du\leqslant 
	\begin{cases}
	n^{p/2}, &p\in[0,2]\\
	(n+p)^{p/2}, &p>2
	\end{cases}
	\end{align*}
\end{lemma}

\begin{lemma}[Lemma 2 from \cite{Nesterov2015UGM}] \label{lemma_2_from_Nesterov_2015_UGM} Let the function $f$ satisfy Assumption \ref{asmpt:holder_grad_function}. Then for any $\tilde{\delta}>0$
	\begin{align*}
    \frac{L_{\nu}}{1+\nu}t^{1+\nu}\leqslant \frac{1}{2}\ls\frac{1-\nu}{1+\nu}\frac{2}{\tilde{\delta}}\rs^{\frac{1-\nu}{1+\nu}}L_{\nu}^{\frac{2}{1+\nu}}t^2+\tilde{\delta} = \frac{L}{2}t^2+\tilde{\delta}
    \end{align*}
\end{lemma}

\begin{theorem}[Theorem 3 from \cite{NesterovSpokoiny2015}] \label{theorem_3_from_Nesterov_Spokoiny} If $f$ is differentiable at $x$ and $u$ is a standard random normal vector, then
    \begin{align*}
        \Exp_{u}\ls\la\nabla f(x), u\ra^2 \|u\|^2\rs \leqslant (n+4)\|\nabla f(x)\|_{\ast}^2
    \end{align*}
\end{theorem}

\end{document}